\documentclass{article}
\usepackage{prelude}

\newkeytheorem{definition}[name = Definition, numbered=unless unique, numberwithin=section]
\newkeytheorem{example}[name = Example, style = example, numberwithin=section]
\newkeytheorem{theorem}[name = Theorem, style = theorem, numberwithin=section]
\newkeytheorem{proposition}[name = Proposition, style = theorem, numberwithin=section]
\newkeytheorem{formula}[name = Formula, style = theorem]
\newkeytheorem{corollary}[name = Corollary, style = theorem, numberwithin=section]
\newkeytheorem{lemma}[name = Lemma, style = theorem, numberwithin=section]
\newkeytheorem{proof}[name = Proof, style = proof, numbered=no]
\newkeytheorem{remark}[name = Remark, style = remark]

\newcommand{\E}{\mathcal E}
\newcommand{\str}[1]{\mathcal{S}^{\pa{#1}}}
\newcommand{\Lag}[1]{L^{\pa{#1}}}
\newcommand{\LL}{\mathfrak L}
\newcommand{\qbinom}[2]{\binom{#1}{#2}_{\mkern-7mu q}}

\DeclareMathOperator{\G}{G}
\DeclareMathOperator{\itlog}{itlog}

\title{On formulas and fractional exponents for umbral operators}
\author{Kei Beauduin}
\date{}

\begin{document}

\maketitle

\begin{abstract}
    We present a new formula for umbral operators that yields three main insights. First, it makes explicit a connection between umbral calculus and iteration theory. Second, it leads naturally to a definition of fractional exponents of umbral operators. Third, its proof synthesizes a broad range of existing results in operational calculus and highlights their combined effectiveness. As an illustration, we obtain a new and natural extension of the Laguerre polynomials.
\end{abstract}

\section{Introduction}

Previous work on \emph{Laguerre polynomials} \cite{dattoli1998,dattoli1999,dattoli2000,dattoli2008} led to an explicit expression for the \emph{Laguerre umbral operator} $L$, namely the linear map sending the monomials $\{x^n\}_{n\ge0}$ to the Laguerre polynomials $\{L_n(x)\}_{n\ge0}$. More precisely, this research shows that $L = e^{-\x D^2}$, where $\x$ denotes multiplication by $x$ and $D$ differentiation with respect to $x$.

This observation motivates the present study: can a general umbral operator $\phi$ be expressed in the form
\begin{equation}\label{e:exVintro}
    \phi = e^{\x V(D)} = 1 + \x V(D) + \inv2 (\x V(D))^2 + \ldots,
\end{equation}
where $V$ is a formal power series? We answer this question completely, and our proof synthesizes and reveals the strength of two centuries of developments in operational calculus.

In \Cref{s:opcalc} we introduce the notation and preliminary facts from operational calculus that we will need, and in \Cref{s:known} we review several known expressions for $\phi$. In \Cref{s:proof} we assemble a number of existing results to derive our formula for $V$. In particular, we recall the minimal amount of \emph{iteration theory} needed to express $V$ in terms of the \emph{iterative logarithm}. The connection between umbral calculus and iteration theory appears to be largely unexplored; to the best of our knowledge, it has been made explicit only in the work of Labelle \cite{labelle1980}.

In \Cref{s:converse} we prove the converse statement: given $V$, we show that \cref{e:exVintro} indeed defines an umbral operator. Then, in \Cref{s:frac}, we define and study fractional exponents of umbral operators. Finally, in \Cref{s:laguerre}, we illustrate the theory with a new family of generalized Laguerre polynomials.

\section{Operational calculus preliminaries}\label{s:opcalc}

Let $\Hom(V, W)$ denote the set of $\C$-linear maps between two vector spaces $V$ and $W$, and let $\End(V) := \Hom(V, V)$ denote the set of endomorphisms of $V$. Let $\C[x]$ and $\C\bbra{x}$ denote, respectively, the sets of polynomials and formal power series in the variable $x$. Any element of $\Hom(\C[x], \C\bbra{x})$ will be called an \emph{operator}.

\subsection{Pincherle derivative}

The \emph{Pincherle derivative} $U'$ \cite{pincherle1895,rota1973,beauduin2024} of an operator $U$ is defined by
\begin{equation}\label{e:Pderiv}
    U' := U \x - \x U.
\end{equation}
The Pincherle derivative is $\C\bbra{\x}$-linear and, for any power series $g$, satisfies $g(D)' = g'(D)$. In other words, for shift-invariant operators, it behaves like the usual derivative, but with respect to $D$. \Cref{e:Pderiv} extends to the $n$-th Pincherle derivative as follows \cite{pincherle1895}:
\begin{equation}\label{e:U(n)}
    U^{(n)} = \sum_{k=0}^n \binom{n}{k} (-\x)^{n-k} U \x^k,
\end{equation}
This identity is readily proved by induction and can be inverted to give
\begin{equation}\label{e:Uxn}
    U \x^n= \sum_{k=0}^n \binom{n}{k} \x^{n-k} U^{(k)}.
\end{equation}

\subsection{Generalized exponentiation}

We will use the generalized exponentiation of operators $U$ and $V$, defined by
\begin{equation}\label{e:U^V}
    U^V := \sum_{k=0}^\infty (U-1)^k \binom{V}{k},
\end{equation}
where $\binom{x}{k} := \frac{(x)_k}{k!}$ denotes the \emph{generalized binomial coefficient}, and $(x)_k := x(x-1) \cdots (x-k+1)$ is the \emph{falling factorial}. It is important to note that, in general,
\begin{equation}\label{e:U^Vrev}
    U^V \neq \sum_{k=0}^\infty \binom{V}{k} (U-1)^k,
\end{equation}
since operators need not commute.\footnote{The right-hand sides of \cref{e:U^Vrev} and \cref{e:U^V} coincide, respectively, with the right- and left-hand power operations introduced by Hines in \cite{hines1955}. The present definition was introduced by the present author in \cite{beauduin2025}.} In \cite[Prop.~2.2]{beauduin2025}, the following identity was established:
\begin{equation}\label{e:eUV}
    \sum_{k=0}^\infty \inv{k!} U^k V^k = (e^U)^V.
\end{equation}
Again, $(e^U)^V = e^{UV}$ holds only when $U$ and $V$ commute. This identity lets us express, in compact form, a type of series that appears frequently in operational calculus.

\subsection{Generating function of an operator}

To any operator $U$, we associate the sequence of polynomials $(U_n(x))_{n\ge0}$ defined by $U_n(x) := U x^n$. Its generating function $\G_U$ is
\begin{equation}\label{e:G}
    \G_U(x, t) := \sum_{n=0}^\infty \frac{U_n(x)}{n!} t^n.
\end{equation}
A sequence of polynomials $(p_n(x))_{n\ge0}$ such that $\deg p_n(x)=n$ for every nonnegative integer $n$ is called a \emph{polynomial sequence}, and is denoted by $\{p_n(x)\}_{n\ge0}$. Any such sequence determines an invertible operator $p$ satisfying $p x^n = p_n(x)$.

\subsection{Umbral calculus}

For the rest of this paper, we let $f$ be a compositionally invertible formal power series, meaning $f(0) = 0$ and $f'(0) \neq 0$. It can be expressed as
\begin{equation}
    f(t) = \sum_{n=1}^\infty \frac{t^n}{n!} a_n, \quad a_1 \neq 0.
\end{equation}
Such a series defines a sequence of polynomials $(\phi_n(x))_{n\ge0}$ through the generating function
\begin{equation}\label{e:genphi}
    e^{xf(t)} = \sum_{n=0}^\infty \frac{t^n}{n!} \phi_n(x).
\end{equation}
This sequence is characterized by the following properties: $\{\phi_n(x)\}_{n\ge0}$ is a polynomial sequence, $\phi_0(x) = 1$, $\phi_n(0) = 0$ for $n>0$, and there exists a \emph{delta operator} $Q$ such that, for all $n>0$,
\[
Q \phi_n(x) = n \phi_{n-1}(x).
\]
The operator $Q$ is unique and satisfies
\begin{equation}
    Q = f^{-1}(D).
\end{equation}
The associated operator $\phi$ is called an \emph{umbral operator}. Since it satisfies $\G_\phi(x, t) = e^{xf(t)}$, we say that $\phi$ is \emph{generated} by $f$.

\section{Known expressions for umbral operators}\label{s:known}

Finding closed forms for umbral operators is challenging. Aside from the identity, such operators are not shift-invariant, so the expansion theorem \cite{rota1973,beauduin2024} does not apply directly. Nevertheless, several noteworthy explicit formulas do appear in the literature, and we collect them in this section.

\subsection{A formula of Garsia and Joni}

The first formula ever published is due to Garsia and Joni \cite{garsia1977}:
\begin{formula}\label{f:garsia}
    \begin{equation}\label{e:garsia}
        \phi = \sum_{k=0}^\infty \frac{\x^k}{k!} \E f(D)^k,
    \end{equation}
\end{formula}
where $\E f(x) := f(0)$ is the operator of evaluation at the origin. This result may be viewed as an extension of Taylor expansion in a direction different from that of the expansion theorem \cite{rota1973,beauduin2024}.
\begin{proof}
    Let $\psi := \phi^{-1}$. According to \cite[Thm.~3.3]{beauduin2024}, $\psi$ is the umbral operator associated with the delta operator $f(D)$. Expanding $E^y$ into powers of $f(D)$ via the expansion theorem \cite[Thm.~3.1]{beauduin2024}, and then applying the result to a polynomial $p(x)$, yields
    \begin{equation}\label{e:expansion}
        p(x+y) = \sum_{n=0}^\infty \frac{\psi_n(y)}{n!} f(D)^k p(x).
    \end{equation}
    We use \cref{e:expansion} in the following computation, where subscripts such as $x$ and $y$ indicate the variable on which the operator acts:
    \begin{align}\label{e:cor7}
        \phi p(x)
        &= \phi_x \E_y p(x+y) =  \phi_x \E_y \sum_{n=0}^\infty\frac{\psi_n(x)}{n!} f(D)_y^k p(y) \nonumber\\
        &= \sum_{n=0}^\infty \frac{x^n}{n!} \E_y f(D)_y^k p(y) = \sum_{n=0}^\infty \frac{\x^n}{n!} \E f(D)^k p(x). 
    \end{align}
    This proves the equality of the operators.
\end{proof}

Some credit for \Cref{f:garsia} is also due to Rota and his collaborators, since \cref{e:cor7} appears as \cite[Cor.~7, p.~711]{rota1973}.

\subsection{Operational form of the Steffensen formulas}

The second and third formulas follow directly from well-known closed forms for basic sequences due to Steffensen \cite{steffensen1941} (see also \cite{rota1973,beauduin2024}):
\begin{equation}
\phi_n(x) = Q'(D/Q)^{n+1} x^n = x (D/Q)^n x^{n-1}.
\end{equation}
A direct identification of operators is not possible here, since the operators on the right-hand side depend on $n$. However, the identification can be achieved with a simple trick. We begin with the identity $\x D x^n = n x^n$. By linearity, this generalizes to
\begin{equation}\label{e:xD}
g(\x D) x^n = g(n) x^n,
\end{equation}
for any power series $g$ convergent at $n$, an identity due to Boole \cite[eq.~(VI.)]{boole1844}. This leads to the following formulas.
\begin{formula}\label{f:steffensen}
    \begin{equation}\label{e:steffensen}
        \phi = Q'(D/Q)^{\x D+1}.
    \end{equation}
\end{formula}

\begin{example}\label{x:stretch}
    When $Q = D/\lambda$, $\lambda \neq 0$, we call $\str\lambda$ the associated umbral operator. It follows from \Cref{f:steffensen} that $\str\lambda = \lambda^{\x D}$. Additionally, according to \cref{e:xD}, $\str\lambda_n(x) = \lambda^{\x D} x^n = \lambda^n x^n$.
\end{example}

\begin{proof}
    By the definition of exponentiation \cref{e:U^V} and \cref{e:xD},
    \begin{align*}
        Q' (D/Q)^{\x D} x^n
        &= Q' \sum_{k=0}^\infty (D/Q - 1)^k \binom{\x D+1}{k} x^n \\
        &= Q' \sum_{k=0}^\infty (D/Q - 1)^k \binom{n+1}{k} x^n \\
        &= Q' (D/Q)^{n+1} x^n  = \phi_n(x) = \phi x^n,
    \end{align*}
    Hence we may identify the operator with $\phi$.
\end{proof}

The following is shown identically.

\begin{formula}
    \begin{equation}
        \phi = \x (D/Q)^{\x D} \x^{-1}.
    \end{equation}
\end{formula}

\begin{remark}
    A brief remark about exponentiation is in order. We only have
    \begin{equation}
    (D/Q)^{\x D+1} = \frac{D}{Q} (D/Q)^{\x D},
    \end{equation}
    and generally not "$(D/Q)^{\x D+1} = (D/Q)^{\x D} \frac{D}{Q}$". This is a consequence of the convention adopted in \cref{e:U^V}. Although the choice is ultimately arbitrary, it is convenient here, since the alternative convention \cref{e:U^Vrev} would lead to the less elegant closed form "$\phi = Q' (e^{\x D+1})^{\log(D/Q)}$".
\end{remark}

\subsection{Pincherle's contributions}\label{s:pincherle}

The expansion theorem provides a way to expand shift-invariant operators in terms of the derivative $D$. For non-shift-invariant operators, one may ask whether an analogous expansion in terms of $\x$ and $D$ exists, and, if so, what form it takes. These questions were in fact answered by Pincherle in 1895 \cite{pincherle1895} (see also \cite[Chap.~II.II]{pincherle1897} and \cite[Chap.~6]{pincherle1901}). His argument relies on the following theorem, which extends the Leibniz rule to arbitrary operators by means of the Pincherle derivative.

\begin{theorem}[Pincherle's formula]\label{t:pincherle}
    If $U$ is an operator and $p$ is a formal power series, then
    \begin{equation}\label{e:leibniz}
        U p(\x)= \sum_{k=0}^\infty \frac{p^{(k)}(\x) U^{(k)}}{k!}.
    \end{equation}  
\end{theorem}
The general Leibniz rule corresponds to the case $U = D^n$. This case extends easily to operators of the form $q(D)$, a fact known since 1848 \cite{hargeave1848}, but \cref{e:leibniz} is far more general.

\begin{proof}
    Let $(a_n)_{n\ge 0}$ be the coefficients of $p$ and $q$ be a polynomial. Then, using \cref{e:Uxn}, we compute
    \begin{align*}
        &U p(\x) q(x) = \sum_{n=0}^\infty a_n U x^n q(x) = \sum_{n=0}^\infty a_n \sum_{k=0}^n \binom{n}{k} \x^{n-k} U^{(k)} q(x) \\
        ={}& \sum_{k=0}^\infty \inv{k!} \pa{\sum_{n=k}^\infty a_n  (n)_k \x^{n-k}} U^{(k)} q(x) = \sum_{k=0}^\infty \frac{p^{(k)}(\x)}{k!} U^{(k)} q(x).
    \end{align*}
\end{proof}

\begin{corollary}[Pincherle]\label{c:pincherle}
    If $U$ is an operator, then
    \begin{equation}\label{e:pincherle}
        U = \sum_{k=0}^\infty \pa{\sum_{j=0}^k \binom{k}{j} (-\x)^{k-j} U_j(\x)} \frac{D^k}{k!}.
    \end{equation}
\end{corollary}

\begin{remark}
    This formula, also due to Pincherle \cite{pincherle1895}, shows that \emph{every} operator in $\End(\C[x])$ can be expressed in terms of $\x$ and $D$, namely
    \begin{equation}
        \End(\C[x]) = \C[\x]\bbra{D}.
    \end{equation}
    An analogous result holds for our set of operators:
    \begin{equation}
        \Hom(\C[x], \C\bbra{x}) = \C\bbra\x\bbra{D}.
    \end{equation}
    Since $[D, \x] = 1$, the operators $D$ and $\x$ generate the Weyl algebra, so $\C[\x][D]$ is isomorphic to it. Consequently, $\Hom(\C[x], \C\bbra{x})$ is a completion of the Weyl algebra, in fact the largest one we can naturally consider. It is important to note, however, that not every such operator extends to an element of $\End(\C\bbra{x})$, as shown by the counterexample in \cite[Ex.~4.7]{beauduin2024}.
\end{remark}

\begin{remark}
    \Cref{t:pincherle} and \Cref{c:pincherle} appear to have largely disappeared from the literature after Pincherle's original work; they are not mentioned in references such as \cite{davis1936,rota1973,bucchianico1996,mansour2015}. To the best of our knowledge, the only notable exception is \cite{mullin1970}.
\end{remark}

\begin{proof}
    According to \cref{e:U(n)},
    \[
    U^{(k)} 1 = \sum_{j=0}^k \binom{k}{j} (-x)^{k-j} U_j(x).
    \]
    Thus, applying \Cref{e:leibniz} to $1$ yields
    \[
    U p(x) = \sum_{k=0}^\infty \frac{p^{(k)}(x)}{k!} (U^{(k)} 1) = \sum_{k=0}^\infty \pa{\sum_{j=0}^k \binom{k}{j} (-\x)^{k-j} U_j(\x)} \frac{D^k}{k!} p(x),
    \]
    which gives \cref{e:pincherle}.
\end{proof}

An immediate corollary is the following formula for $\phi$.
\begin{formula}\label{f:pincherle}
    \begin{equation}
        \phi = \sum_{k=0}^\infty \pa{\sum_{j=0}^k \binom{k}{j} (-\x)^{k-j} \phi_j(\x)} \frac{D^k}{k!}.
    \end{equation}
\end{formula}

A particularly important example is the \emph{composition operator} $C_g$, defined for a formal power series $g$ and a polynomial $p$ by
\begin{equation}
C_g p(x) = p(g(x)),
\end{equation}
which is an operator not necessarily belonging to $\End(\C[x])$. Although the composition operator was introduced by Pincherle \cite{pincherle1895}, the following explicit formula is due to Bourlet \cite[Sec.~V]{bourlet1897b}.

\begin{proposition}[Bourlet]\label{p:compoform1}
    If $g$ is a formal power series, then
    \begin{equation}\label{e:compoform1}
        C_g = (e^{g(\x) - \x})^D.
    \end{equation}
\end{proposition}

\begin{proof}
    Since $C_g x^n = g(x)^n$, \Cref{c:pincherle} gives
    \[
    C_g = \sum_{k=0}^\infty \pa{\sum_{j=0}^k \binom{k}{j} (-\x)^{k-j} g(\x)^j} \frac{D^k}{k!} = \sum_{k=0}^\infty (g(\x) - \x)^k \frac{D^k}{k!}.
    \]
    So \cref{e:compoform1} follows from \cref{e:eUV}.
\end{proof}

\begin{example}\label{x:stretch2}
    The operator $\str\lambda$ from \Cref{x:stretch} is a composition operator, since the identity $\str\lambda x^n = (\lambda x)^n$ extends to polynomials as $\str\lambda p(x) = p(\lambda x)$. To recover the expression from \Cref{x:stretch} using \Cref{p:compoform1}, we only need the identity $\x^n D^n = (\x D)_n$, due to Boole \cite[eq.~(VII.)]{boole1844}, together with \cref{e:U^V}:
    \[
    \str\lambda = (e^{\lambda \x - \x})^D = \sum_{n=0}^\infty \frac{(\lambda - 1)^n}{n!} \x^n D^n = \sum_{n=0}^\infty (\lambda - 1)^n \binom{\x D}{n} = \lambda^{\x D}.
    \]
\end{example}

\subsection{The Kurbanov-Maksimov expansion}

The fact that operators can be expressed in terms of $\x$ and $D$ goes back to Pincherle's discovery of \Cref{c:pincherle}. An alternative construction was later given by Kurbanov and Maksimov \cite{kurbanov1986}, and Di Bucchianico and Loeb subsequently clarified and generalized it using delta operators \cite{bucchianico1996}. We extend their work by formalizing a formula that was used implicitly, but not stated explicitly, in the proof of \cite[Prop.~21]{bucchianico1996}; see our \Cref{f:bucc}.

\begin{theorem}[Kurbanov-Maksimov]\label{t:KM}
    If $U$ is an operator and
    \[
    \frac{\G_{U \phi}(x, t)}{\G_\phi(x, t)} = \sum_{k=0}^\infty g_k(x) h_k(t),
    \]
    for some formal power series $\{g_k(x)\}_{k\ge0}$ and $\{h_k(t)\}_{k\ge0}$, then $U$ can be written as
    \begin{equation}\label{e:KM}
        U = \sum_{k=0}^\infty g_k(\x) h_k(Q).
    \end{equation}
\end{theorem}

\begin{proof}
    Since $\G_{Q\phi}(x, t) = t \G_\phi(x, t)$, if we denote by $V$ the right-hand side of \cref{e:KM}, then by linearity and construction
    \[
    \G_{V\phi}(x, t) = \sum_{k=0}^\infty g_k(x) h_k(t) \G_\phi(x, t) = \G_{U \phi}(x, t).
    \]
    Comparing coefficients of $t^n$, we see that \cref{e:KM} holds on the basis $\{\phi_n(x)\}_{n\ge0}$, and therefore on all polynomials.
\end{proof}

Following Di Bucchianico and Loeb \cite[Prop.~21]{bucchianico1996}, we obtain the following formula for $\phi$.
\begin{formula}[Di Bucchianico--Loeb]\label{f:bucc}
    \begin{equation}\label{e:bucc}
        \phi = (e^\x)^{f(D) - D}.
    \end{equation}
\end{formula}

\begin{proof}
    According to \cref{e:genphi}
    \[
    \frac{\G_\phi(x, t)}{\G_1(x, t)} = e^{x(f(t) - t)} = \sum_{k=0}^\infty \inv{k!} x^k (f(t) - t)^k,
    \]
    hence, by \Cref{t:KM}, applied with the identity operator and $U = \phi$, together with \cref{e:eUV},
    \[
    \phi = \sum_{k=0}^\infty \inv{k!} \x^k (f(D) - D)^k = (e^\x)^{f(D) - D}.
    \]
\end{proof}

Expanding $\G_\phi$ via its definition \cref{e:G} in the proof above leads directly to \Cref{f:pincherle}.

\begin{example}
    With $f(t) = \lambda t$, $\phi = \str\lambda$ and the formula $\str\lambda = \lambda^{\x D}$ follows by a calculation similar to that of \Cref{x:stretch2}.
\end{example}

\begin{remark}
    In \cite[p.~711, Prop.~3]{rota1973}, it is shown that
    \begin{equation}\label{e:PDE}
        \phi' = \x \phi (f'(D) - 1),
    \end{equation}
    after which Rota suggests solving the "Pincherle differential equations" in order to obtain an explicit expression for $\phi$. He later restated this challenge as the 15th problem in \cite[p.~753]{rota1973}, describing it as "an untouched subject of great interest". This perspective now appears more tractable, however, because the solution of the Pincherle differential equation \eqref{e:PDE} is precisely the operator in \Cref{f:bucc}. Indeed, the right-hand side of \cref{e:bucc} also satisfies \cref{e:PDE}, by $\C\bbra{\x}$-linearity and the identity $(f(D) - D)' = f'(D) - 1$. The only missing ingredient is a noncommutative analog of the Cauchy uniqueness theorem.
\end{remark}

\section{A new expression for umbral operators}\label{s:proof}

We now seek to express umbral operators in the simplest possible form:
\begin{equation}\label{e:exV}
    \phi = e^{\x V(D)} = 1 + \x V(D) + \inv2 (\x V(D))^2 + \ldots.
\end{equation}
The existence of such a representation was anticipated by Zeilberger \cite[\S 4.2]{zeilberger1980}. Although his argument was not fully rigorous, it can be adapted into a proof of the existence of $V$.

\begin{proposition}\label{p:Vex}
    There exists a formal power series $V$ such that $V(0) = 0$ and $\phi = e^{\x V(D)}$.
\end{proposition}

\begin{proof}
    By \Cref{c:pincherle}, there exists a sequence of formal power series $(p_k(x))_{k\ge0}$ such that
    \[
    L := \log \phi = \sum_{k=0}^\infty p_k(\x) D^k.
    \]
    Our goal is to show that $p_0(x) = 0$ and that for $k > 0$, $p_k(x)$ is linear. Operationally, the binomial-type property reads $\phi_{x+y} = \phi_x \phi_y$ in $\C[x+y]$ \cite[Prop.~2.3]{beauduin2024}, where the subscript indicates the variable on which $\phi$ acts. Since $L_x$ and $L_y$ commute, translating this identity to $L$ yields
    \[
    e^{L_{x+y}} = e^{L_x} e^{L_y} = e^{L_x + L_y},
    \]
    and therefore $L_{x+y} = L_x + L_y$. We now apply this identity to the monomial $(x+y)^n$ for $n \ge 0$:
    \[
    \sum_{k=0}^n (p_k(x+y) - p_k(x) - p_k(y))\frac{n!}{(n-k)!} (x+y)^{n-k} = 0.
    \]
    By induction, we conclude that for every nonnegative $k$, $p_k(x+y) = p_k(x) + p_k(y)$, which implies the existence of a sequence $(a_k)_{k\ge 0}$ of complex numbers such that $p_k(x)=a_k x$. Defining $V(t) = \sum_{k\ge 0} a_k t^k$, we obtain $\phi = e^{\x V(D)}$. Finally, the condition that $\phi x^n$ remain a polynomial forces $a_0=0$, or equivalently $V(0)=0$.
\end{proof}

\Cref{p:Vex} will not be used in what follows, since we will derive an explicit expression for $V$ directly.

\subsection{The \texorpdfstring{$\LL$}{L} transform}

In \cite[Thm.~2.1]{beauduin2025}, a new endomorphism $\LL$ on the set of operators was introduced. We now characterize the special instance of this transform that is relevant here \cite[Ex.~2.2]{beauduin2025}.

\begin{theorem}\label{t:L}
    There exists a transform $\LL$ defined over the set of operators such that
    \begin{itemize}
        \item $\LL$ is linear.
        \item For all operators $U, V$, $\LL(UV) = \LL(V) \LL(U)$.
        \item $\LL(\x) = D$ and $\LL(D) = \x$.
    \end{itemize}
    In particular, $\LL^{-1} = \LL$, and for any power series $g$, $\LL(g(U)) = g(\LL(U))$.
\end{theorem}

This transform converts operator identities into forms that are sometimes much easier to handle.

\begin{lemma}[{\cite[Ex.~2.2]{beauduin2025}}]\label{l:LCphi}
    \begin{equation}
        \LL(C_f) = \phi.
    \end{equation}
\end{lemma}
This result clearly exhibits a kind of duality between $C_f$ and $\phi$, and raises the question of how the action of $\LL$ should be interpreted.
\begin{proof}
    This result follows from \Cref{p:compoform1}, the properties listed in \Cref{t:L} and \Cref{f:bucc}:
    \begin{align*}
        \LL(C_f) &= \LL((e^{f(\x) - \x})^D) = \sum_{n=0}^\infty \inv{k!} \LL(D^k) \LL((f(\x) - \x)^k) \\
        &= \sum_{n=0}^\infty \inv{k!} \x^k(f(D) - D)^k = (e^\x)^{f(D) - D} = \phi.
    \end{align*}
\end{proof}

\subsection{Iteration theory and proof of the new formula}

The notion of \emph{fractional iteration}, that is, non-integral compositional iterates of a formal power series, can be defined uniquely under suitable hypotheses; see \cite{labelle1980,ecalle1974,ecalle1975}. More precisely, for a formal power series $g$, one seeks a family $(g^s)_{s\in\C}$ such that $g^1 := g$ and
\[
g^r \circ g^s = g^{r+s}
\]
for all $r, s \in \C$. This is the central object of \emph{iteration theory}. The \emph{iterative logarithm} $\itlog$ is the natural analog of the ordinary logarithm in the setting where composition plays the role of multiplication. For instance, it satisfies $\itlog(g^s) = s \itlog(g)$ and is defined by
\begin{equation}
    \itlog(g) := \odv{g^s}{s}_{s=0}.
\end{equation}

Bourlet was the first to apply composition operators to iteration theory \cite{bourlet1898}. \'Ecalle later reintroduced this point of view and proved the following formula \cite{ecalle1970,ecalle1971,ecalle1974,ecalle1975}.

\begin{proposition}\label{p:ecalle}
    For an invertible formal power series $g$ and $s\in\C$,
    \begin{equation}\label{e:Citlog}
        C_{g^s} = e^{s \itlog(g)(\x) D}.
    \end{equation}
\end{proposition}

We follow the proof of Labelle \cite{labelle1980}.

\begin{proof}
    For some power series $p$, first we compute via the chain rule
    \[
    \begin{gathered}
        \odv*{p(g^s(x))}{s} = \odv*{p(g^{s+r}(x))}{s}_{r=0} = \odv*{p(g^{s+r}(x))}{r}_{r=0} = \odv*{p(g^s(g^r(x)))}{r}_{r=0} \\
        = \odv{g^r(x)}{r}_{r=0} \odv{p(g^s(x))}{x} = \itlog(g)(x) \odv{p(g^s(x))}{x} = \itlog(g)(\x) D p(g^s(x)).
    \end{gathered}
    \]
    This is a differential equation in $s$, with solution
    \[
    p(g^s(x)) = e^{s\itlog(g)(\x) D} p(x),
    \]
    hence \cref{e:Citlog}.
\end{proof}

This leads directly to our main result.

\begin{formula}\label{f:Eform}
    \begin{equation}
        \phi = e^{\x\itlog(f)(D)}.
    \end{equation}
\end{formula}

\begin{proof}
We simply combine the \Cref{l:LCphi} with \Cref{p:ecalle}:
\[
\phi = \LL(C_f) = \LL(e^{\itlog(f)(\x) D}) = e^{\x \itlog(f)(D)}.
\]
\end{proof}

\begin{example}
    If $f(t) = \lambda t$, then $f^s(t) = \lambda^s t$ and
    \[
    \itlog(f)(t) = \odv{\lambda^s t}{s}_{s=0} = \log(\lambda) t.
    \]
    Therefore, with \Cref{f:Eform}, we find once again that $\str\lambda = \lambda^{\x D}$.
\end{example}

\subsection{Converse of the theorem}\label{s:converse}

We have expressed $V$ in terms of $f$, but the reverse direction is also possible: given $V$, we can find $f$ such that $e^{\x V(D)}$ is the umbral operator associated with $f$.

Our first ingredient is the well-known operational form of Taylor's theorem due to Lagrange in 1772 \cite{lagrange1772}. For a polynomial $p$, if $E^a p(x) := p(x+a)$, then
\begin{equation}\label{e:shift}
    E^a = e^{aD}.
\end{equation}
This identity was extended by Reverend Graves in 1850 \cite{graves1850}.

\begin{proposition}\label{p:compoform2}
    Let $V$ be a formal power series and let $h$ be a function satisfying $h'(t) = 1/V(t)$. Then $g(t) := h^{-1}(h(t) + 1)$ defines a formal power series satisfying $g^s(t) = h^{-1}(h(t) + s)$, for all $s\in\C$, and $\itlog(g) = V$. In addition,
    \begin{equation}
        e^{s V(\x)D} = C_{g^s}.
    \end{equation}
\end{proposition}

Here the definition of $h$ is meant only formally. In particular, when $V$ is a formal power series, $h$ need not belong to the same class of functions, and expressions such as $h^{-1}(h(t)+s)$ should be understood as formal notation for the flow associated with $V$.

\begin{proof}
    The functions $G(s, x) := h^{-1}(h(x) + s)$ satisfy the functional equation $G(s, G(r, x)) = G(s+r, x)$. Moreover, $\odv*{G(s, x)}{s}_{s=0} = V(x)$, so by an argument similar to the proof of \Cref{p:ecalle}, we obtain $G(s, x) = e^{s V(\x) D} x$. Since $V$ is a formal power series, so is $g(x) = G(1, x)$. By uniqueness of the fractional iterates of the formal power series $g$, it follows that $g^s(x) = G(s, x)$ and therefore $\itlog(g) = V$.

    Let $p$ be a polynomial and $P(x) = p(h^{-1}(x))$, then by the chain rule $D_x = h'(\x) D_{h(x)}$. Thus, according to \cref{e:shift}
    \begin{align*}
        e^{sV(\x)D} p(x) &= e^{s V(\x) h'(\x) D_h} P(h) = E_h^s P(h) = P(h+s) \\
        &= p(h^{-1}(h(x) + s)) = C_{g^s} p(x).
    \end{align*}
\end{proof}

Combining \Cref{p:compoform1,p:compoform2}, we obtain the identity $e^{V(\x)D} = (e^{g(\x) - \x})^D$, which can also be obtained from a theorem of Viskov \cite{viskov1997a}; see \cite[Cor.~5.2]{beauduin2025}.

\begin{corollary}\label{c:V}
    For a power series $V$, if $h$ is a function such that $h'(t) = 1/V(t)$, then for all $s\in\C$
    \[
    e^{s \x V(D)} = (e^\x)^{h^{-1}(h(D) + s) - D}.
    \]
    In particular, if $V(0) = 0$, $e^{s \x V(D)}$ is the umbral operator generated by $g^s(t) := h^{-1}(h(t) + s)$.
\end{corollary}

\begin{proof}
    We combine \Cref{p:compoform1,p:compoform2} and apply $\LL$:
    \begin{align*}
        \LL(e^{sV(\x)D}) &= \LL((e^{h^{-1}(h(\x) + s) - \x})^D), \\
        e^{s\LL(V(\x)D)} &= (e^{\LL(D)})^{\LL(h^{-1}(h(\x) + s) - \x)}, \\
        e^{s\x V(D)} &= (e^\x)^{h^{-1}(h(D) + s) - D}.
    \end{align*}
    If $V(0) = 0$, then $e^{s\x V(D)}$ is an endomorphism of the polynomials, and since $e^{s\x V(D)} = \LL(C_{g^s})$, \cref{l:LCphi} shows that it is the umbral operator generated by $g^s(t)$.
\end{proof}

\section{Fractional exponents of umbral operators}\label{s:frac}

In light of \Cref{f:Eform}, there is a natural way to define fractional exponents of an umbral operator:
\begin{equation}
    \phi^s := e^{s\x \itlog(f)(D)}.
\end{equation}
Since $s \itlog(f) = \itlog(f^s)$, the operator $\phi^s$ is generated by the formal power series $f^s$. This is precisely the definition adopted in \cite{beauduin2025b}. It contrasts with an alternative definition based on the binomial series, which involves convergence issues.

In \cite{beauduin2025b} we derived several formulas for the coefficients of $f^s(t)^k$, which in turn determine the coefficients of $\phi^s x^n =: \phi^s_n(x)$. For instance, if we let $q := f'(0)$, then the coefficient $\coeff{n}{k}_{\phi^s}$ of $x^k$ in the polynomial $\phi^s_n(x)$ is given by \cite[Thm.~7]{beauduin2025b}
\begin{equation}
    \coeff{n}{k}_{\phi^s} = \sum_{p=0}^{n-k} \coeff{n}{k}_{\phi^p} \qbinom{s}{p} \qbinom{n-k-s}{n-k-p} q^{(n-p)(s-p)},
\end{equation}
where we use the \emph{$q$-binomial coefficient}. In the same work, we also derived an expansion for the iterative logarithm of $f$. Its coefficients are given by $\odv*{\coeff{n}{1}_{\phi^s}}{s}_{s=0}$, which may be interpreted as the coefficient $\coeff{n}{1}_{\log\phi}$, since $\log \phi = \x \itlog(f)(D)$ by \Cref{f:Eform}. This logarithm is the infinitesimal generator of the one-parameter group $(\phi^s)_{s\in\C}$, which satisfies
\begin{equation}
\phi^{s+t} = \phi^s \phi^t.
\end{equation}
In addition, the delta operator $Q^{[s]} := f^{-s}(D)$ associated with $\phi^s$ satisfies
\begin{equation}
Q^{[t]}\phi^s = \phi^s Q^{[t-s]},
\end{equation}
for all $s, t\in\C$, extending \cite[Prop.~2.8]{beauduin2024}. Furthermore, $(Q^{[s]})_{s\in\C}$ itself forms a one-parameter group under the law
\begin{equation}
Q^{[s+t]} = Q^{[s]} \diamond Q^{[t]} := (f^{-s} \circ f^{-t})(D),
\end{equation}
with infinitesimal generator $-\itlog(f)(D)$. Thus, the preceding results provide a complete answer to Rota's third problem \cite[Sec.~14]{rota1973}, now in the more general setting where $f'(0)$ need not equal $1$.

\section{Example: degenerate Laguerre polynomials}\label{s:laguerre}

We apply the construction of \Cref{c:V} with $V(t) = -t^{p+1}$, where $p\geq 0$ is an integer. Let $L_p := e^{-\x D^{p+1}}$. By \Cref{x:stretch}, we have $L_0 = \str{1/e}$, so from now on we restrict to the case $p > 0$. Let $h(t) := \frac{1}{pt^p}$, which satisfies $h'(t) = 1/V(t)$. Then
\begin{equation}\label{e:f^s}
    f^s(t) := h^{-1}(h(t) + s) = \sqrt[-p]{p\pa{\frac{1}{p t^p} + s}} = \sqrt[-p]{\frac{1 + sp t^p}{t^p}} = \frac{t}{\sqrt[p]{1 + sp t^p}}.
\end{equation}
The corresponding delta operator is therefore $\Psi_p := f^{-1}(D) = D/\sqrt[p]{1-pD^p}$. The case $p = 1$ corresponds to the classical \emph{Laguerre polynomials} \cite[Sec.~6.5]{beauduin2024}.

We refer to $L_{p, n}(x) := (L_p)_n(x)$ as the \emph{degenerate Laguerre polynomials}, and define their \emph{associated} version by $\Lag\alpha_p := (1-pD^p)^{\frac{\alpha}{p}} L_p$. They form an example of a \emph{cross-sequence}, so by \cite[Prop.~3.6]{beauduin2024}, for all $\alpha, \beta$,
\begin{equation}
    \Lag{\alpha + \beta}_{p,n}(x+y) = \sum_{k=0}^n \binom{n}{k} \Lag\alpha_{p,k}(x) \Lag\beta_{p,n-k}(y),
\end{equation}
and their generating function is given by \cite[Thm.~4.2]{beauduin2024}
\begin{equation}
    \sum_{n=0}^\infty \Lag\alpha_{p, n}(x) \frac{t^n}{n!} = \inv{(\sqrt[p]{1 + p t^p})^\alpha} \exp\pa{\frac{xt}{\sqrt[p]{1+p t^p}}}.
\end{equation}

$\Psi_p' = (1-pD^p)^{-\inv{p}-1}$, thus, with the transfer formula \cite[Thm.~3.4]{beauduin2024} we have
\[
L_{p,n}(x) = \Psi_p' (\sqrt[p]{1-pD^p})^{n+1} x^n = (1-pD^p)^{\frac{n}{p} - 1} x^n,
\]
therefore
\begin{align}\label{e:lagalpha}
    \Lag\alpha_{p,n}(x) &= (1-pD^p)^{\frac{n+\alpha}{p} - 1} x^n = \sum_{k=0}^\infty \binom{\frac{n+\alpha}{p} -1}{k} (-p D^p)^k x^n \nonumber\\
    &= \sum_{k=0}^{\floor{n/p}} \binom{\frac{n+\alpha}{p} - 1}{k} \frac{n! (-p)^k}{(n-pk)!} x^{n-pk}.
\end{align}

A simple expression for the logarithm of the operator $\Lag\alpha_p$ follows from the identity
\begin{equation}\label{e:LagUO}
    \Lag\alpha_p = \exp(-\x D^{p+1} - \alpha D^p).
\end{equation}
The case $p=\alpha = 1$, corresponding to a different convention for the Laguerre polynomials, was discovered by Dattoli and Torre \cite{dattoli1998}; see also \cite{dattoli1999,dattoli2000}. Dattoli and Migliorati later extended this result to arbitrary $\alpha$ \cite[eq.~(33) with $f(x) = x^n$ and $y = 1$]{dattoli2008}. Formula \eqref{e:LagUO} then follows quickly from the \emph{Berry identity} \cite{berry1966,beauduin2025,mansour2015}: since the commutator satisfies $[D^p, \x D^{p+1}] = p (D^p)^2$, it follows from \cite[Cor.~3.4 and Prop.~2.4]{beauduin2025} that
\[
\exp(-\x D^{p+1} - \alpha D^p) = ((1-p D^p)^{1/p})^\alpha e^{-\x D^{p+1}} = \Lag\alpha_p.
\]
The degenerate associated Laguerre polynomials $\Lag\alpha_{p,n}(x)$ satisfy the differential equation
\begin{equation}\label{e:Lpde}
   x p F^{(p+1)} + \alpha p F^{(p)} - x F' + n F = 0,
\end{equation}
which follows from a computation similar to that in \cite[eq.~(59)]{beauduin2024}.

The Laguerre polynomials also provide a convenient example for real exponents. Using the expression for $f^s(t)$ from \cref{e:f^s}, we obtain $\Psi_p^{[s]} = D/\sqrt[p]{1-spD^p}$, and a derivation similar to \cref{e:lagalpha} gives
\begin{equation}
(1-spD^p)^{\frac\alpha{p}} L^s_{p, n}(x) = \sum_{k=0}^{\floor{n/p}} \binom{\frac{n+\alpha}{p} - 1}{k} \frac{n! (-sp)^k}{(n-pk)!} x^{n-pk}.
\end{equation}
For $p=\alpha=1$, this formula reduces to the bivariate Laguerre polynomials of Dattoli and Torre \cite{dattoli1998}, with $s$ as the second variable.

\printbibliography

\end{document}